\documentclass[a4paper]{amsart}
\usepackage{amsmath}
\usepackage{amssymb}
\usepackage{bbm}
\usepackage{a4wide}
\usepackage{graphicx}
\usepackage{hyperref}

\newtheorem{thm}{Theorem}
\newtheorem{prop}{Proposition}

\newtheorem{cor}{Corollary}

\newtheorem{conj}{Conjecture}

\begin{document}

\author[Brian Nakamura]{Brian Nakamura$^1$}
\thanks{Mathematics Department, Rutgers University-New Brunswick, Piscataway, NJ, USA. [bnaka@math.rutgers.edu]}
\thanks{Current website: \url{http://math.rutgers.edu/\~bnaka/CAV}}

\title{Computational Approaches to Consecutive Pattern Avoidance in Permutations}

\begin{abstract}
	In recent years, there has been increasing interest in consecutive pattern avoidance in permutations. In this paper, we introduce two approaches to counting permutations that avoid a set of prescribed patterns consecutively. These algoritms have been implemented in the accompanying Maple package {\tt CAV}, which can be downloaded from the author's website. As a byproduct of the first algorithm, we have a theorem giving a sufficient condition for when two pattern sets are strongly (consecutively) Wilf-Equivalent. For the implementation of the second algorithm, we define the cluster tail generating function and show that it always satisfies a certain functional equation. We also explain how the {\tt CAV} package can be used to approximate asymptotic constants for single pattern avoidance.
\end{abstract}

\maketitle

\section{Introduction}

Let $\sigma = \sigma_{1} \cdots \sigma_{k}$ be a sequence of $k$ distinct positive integers. We define the \emph{reduction} $\rho(\sigma)$ to be the length $k$ permutation we get by relabeling the elements of $\sigma$ with $\{ 1, \ldots, k \}$ so that they retain the same order relationships they had in $\sigma$. For example, $\rho(5 3 8 6) = 2 1 4 3$. For a permutation $p$, we will also write $|p|$ for the number of elements in the permutation. Let $m$ and $n$ be positive integers with $m \leq n$, and let $p \in \mathcal{S}_{m}$ and $\sigma = \sigma_{1} \cdots \sigma_{n} \in \mathcal{S}_{n}$. We will say that $\sigma$ \emph{contains} $p$ \emph{consecutively} if $\rho(\sigma_{i} \cdots \sigma_{i+m-1}) = p$ for some $i$ where $1 \leq i \leq n-m+1$. Otherwise, we say that $\sigma$ \emph{avoids} $p$ \emph{consecutively}. Similarly, if $B$ is a set of permutations, then we say that $\sigma$ \emph{avoids} $B$ \emph{consecutively} if for every $p \in B$, $\sigma$ avoids the pattern $p$ consecutively. For example, the permutation $1 2 3 6 5 4 \in \mathcal{S}_{6}$ contains the permutation pattern $1 2 4 3$, since $\rho(2 3 6 5) = 1 2 4 3$. However, the permutation $1 2 4 5 3 \in \mathcal{S}_{5}$ avoids the pattern $1 2 4 3$ consecutively.\\

In general, we are interested in counting permutations that avoid a pattern (or a set of patterns). Given a set of patterns $B$, let $\alpha_{B}(n)$ be the number of length $n$ permutations that avoid $B$ consecutively. If $B$ consists of only a single pattern $p$, we may write $\alpha_{p}(n)$ instead, and if no ambiguity would arise, we may just write $\alpha(n)$. For a given set of patterns $B$, we would like to find the exponential generating function

\begin{eqnarray}
	A_{B}(z) = \mathop{\sum} \limits_{n=0}^{\infty} {\alpha_{B}(n) \frac{z^{n}}{n!}}.
\end{eqnarray}

\noindent If no ambiguity would arise, this may also be denoted by $A(z)$. In addition, we define a more general exponential generating function

\begin{eqnarray}
	P_{B}(z,t) = \mathop{\sum} \limits_{k,n \geq 0} {b_{k,n} \frac{z^{n} t^{k}}{n!}}
\end{eqnarray}

\noindent where $b_{k,n}$ is the number of length $n$ permutations that contain exactly $k$ occurrences of the patterns in $B$. Again, we may write $P(z,t)$ if the set $B$ is clear. We will also define $\alpha(n,t) = \sum_{k \geq 0}{b_{k,n} t^{k}}$. Note that $P(z,0) = A(z)$ and $\alpha(n,0) = \alpha(n)$.\\

In addition, we will say that two sets of patterns $B$ and $B^{\prime}$ are consecutively Wilf-Equivalent (sometimes written c-Wilf-Equivalent) if $A_{B}(z) = A_{B^{\prime}}(z)$. We will also say that $B$ and $B^{\prime}$ are strongly c-Wilf-Equivalent if $P_{B}(z,t) = P_{B^{\prime}}(z,t)$. Since this paper deals solely with consecutive patterns, the word ``consecutive'' will be omitted in most instances. For the rest of this paper, the reader should assume that all mentions of containment, avoidance, and Wilf-Equivalence are consecutive.\\

In recent years, there has been an increasing amount of research done on consecutive pattern avoidance in permutations. One of the early papers by Elizalde and Noy (~\cite{EN}) finds generating functions $A(z)$ and $P(z,t)$ for certain cases of single pattern avoidance. Using various techniques, additional generating functions for specific single patterns and multi-pattern sets have been found in ~\cite{AAM, DK, LR, MR}. In particular, our approach will resemble the cluster method approach in ~\cite{DK}.\\

The results in this paper utilize an extension of the Goulden-Jackson cluster method (\cite{GJ, NZ}). We restate some of the terminology and notation here.\\

Let $B$ be a set of patterns. Without loss of generality, assume that $B$ contains no trivial redundancies (i.e., there are no $p_{1},p_{2} \in B$ with $p_{1} \neq p_{2}$ such that $p_{1}$ contains $p_{2}$). We say that an ordered pair $(p; [[i_{1},j_{1}], \ldots, [i_{m},j_{m}]])$ is a length $k$ cluster if it satisfies the following:\\

\begin{enumerate}
	\item[(a)] $p \in \mathcal{S}_{k}$
	\item[(b)] $i_{1} = 1$, $j_{m} = k$, and $i_{n} < i_{n+1} < j_{n}$ for $1 \leq n \leq m-1$ (i.e., each interval overlaps with the neighboring interval, and the intervals cover $p$)
	\item[(c)] $\rho(p_{i_{n}} \cdots p_{j_{n}}) \in B$ for all $1 \leq n \leq m$.\\
\end{enumerate}

\noindent Let $\mathcal{C}_{k}$ be the set of clusters of length $k$, and for a cluster $w = (p; [[i_{1},j_{1}], \ldots, [i_{m},j_{m}]])$, define $weight(w) = (t-1)^{m}$, where $t$ will be the variable used to track occurrences. Let $C(k) = \mathop{\sum} \limits_{w \in \mathcal{C}_{k}} {weight(w)}$. From an adaptation of ~\cite{GJ} to the present context of an ``infinite'' alphabet and exponential generating functions, we have:

\begin{thm}\label{ThmGJ}
	\begin{eqnarray}
		P(z,t) = \frac{1}{1 - z - \mathop{\sum} \limits_{k \geq 1} {C(k) \frac{z^{k}}{k!}}}
	\end{eqnarray}
\end{thm}

\noindent This theorem will essentially be the basis for our algorithm and our results.\\

So far, generating functions have been found for specific single patterns and multi-pattern sets and for certain single pattern families where some specific structure can be exploited. In this paper, we will outline two algorithms to calculate $\alpha(n)$ more efficiently, and both algorithms have been implemented in the accompanying Maple package {\tt CAV}. The Maple package can be downloaded from the author's website. As a result of the first algorithm in Section~\ref{SectCAV}, we get a theorem for proving when two pattern sets are strong c-Wilf-Equivalent. During preparation of this paper, the author learned that this result was also independently proven by Khoroshkin and Shapiro in \cite{KS} by slightly different means. To establish the much faster second algorithm in Section~\ref{SectCAVT}, we define a new generating function which we refer to as the cluster tail generating function. We show that this generating function always satisfies a certain functional equation and give a constructive approach to finding it. This functional equation is then used to compute values for $\alpha(n)$ much more quickly. We use our algorithm to give some asymptotic approximations in Section~\ref{SectAsym}. We conclude with Section~\ref{SectConj} by sharing some new conjectures we have based off of experimentation with our {\tt CAV} package. Beyond the theorems and results in this paper though, we hope that the {\tt CAV} Maple package will be a useful tool for others in studying consecutive pattern avoidance in permutations. \\

\section{Consecutive Pattern Avoidance via Clusters}\label{SectCAV}

Let $B$ be a set of fixed patterns that we would like to avoid (consecutively). We may assume that $B$ contains no redundancies (i.e., there does not exist $p_{1},p_{2} \in B$ with $p_{1} \neq p_{2}$ such that $p_{1}$ contains $p_{2}$). Again, $\alpha(n)$ will be the number of length $n$ permutations avoiding $B$.\\

From the Goulden-Jackson cluster method (Theorem~\ref{ThmGJ}), we can get the equation

\begin{eqnarray}
	P(z,t) = 1 + z P(z,t) + P(z,t) \mathop{\sum} \limits_{k \geq 1} {C(k) \frac{z^{k}}{k!}}
\end{eqnarray}

\noindent and by extracting the coefficients of $z^{n}$, we get the following recurrence:

\begin{eqnarray}
	\alpha(n,t) = n \alpha(n-1,t) + \mathop{\sum} \limits_{k=1}^{n} { {n \choose k} C(k) \alpha(n-k,t)} \label{AvRecur}
\end{eqnarray}

\noindent Additionally, consider a fixed $p \in B$ and let $m = |p|$. Let $\mathcal{C}_{k}[p] = \{ (\pi; [[i_{1},j_{1}], \ldots, [i_{r},j_{r}]]) \in \mathcal{C}_{k} : \rho(\pi_{i_{r}} \cdots \pi_{j_{r}}) = p \}$, the set of length $k$ clusters ending in the pattern $p$. Let $\mathcal{C}_{k}[p; [x_{1}, \ldots, x_{m}]]$ be the clusters in $\mathcal{C}_{k}[p]$ with the last $m$ terms $\{x_{1}, \ldots, x_{m}\}$, where $x_{1} < x_{2} < \ldots < x_{m}$. Similarly, define

\begin{eqnarray}
 C(k,p) & = & \mathop{\sum} \limits_{w \in \mathcal{C}_{k}[p]} {weight(w)}\label{Ckp}\\
 C(k,p;[x_{1}, \ldots, x_{m}]) & = & \mathop{\sum} \limits_{w \in \mathcal{C}_{k}[p; [x_{1}, \ldots, x_{m}]]} {weight(w)} \label{Ckpx}.
\end{eqnarray}

\noindent If $B$ contains only one pattern, these may be denoted by $C(k)$ and $C(k;[x_{1}, \ldots, x_{m}])$, respectively.  We will use Equation \eqref{AvRecur} to count permutations avoiding $B$.\\

\subsection{General Algorithm}\label{CavAlg} 

Computationally, the difficulty in using Equation \eqref{AvRecur} lies in calculating $C(k)$ quickly. One way to do this is to create a recurrence for $C(k,p;[x_{1}, \ldots, x_{|p|}])$ for each $p \in B$.\\

We can do this as follows: for a given cluster $w$, let $p_{1}$ and $p_{2}$ be the last marked pattern and the second to last marked pattern in $w$, respectively. Let $j$ be the length of the overlap of $p_{1}$ and $p_{2}$ in $w$, i.e., the tail of length $j$ of $p_{2}$ coincides with the head of length $j$ of $p_{1}$. We want to ``chop off'' the last $|p_{1}| - j$ terms of $w$ and apply the reduction to get a shorter cluster, say $w^{\prime} \in \mathcal{C}_{k^{\prime}}[p_{2}; [x_{1}, \ldots, x_{|p_{2}|}]]$, which ends in the pattern $p_{2}$. Then, $weight(w) = weight(p_{1}) \cdot weight(w^{\prime})$.\\

Additionally, once we ``chop off'' the tail of $p_{1}$ and apply the reduction to get a shorter cluster $w^{\prime}$, we actually know what the last $j$ terms of $w^{\prime} = w_{1}^{\prime} \cdots w_{k^{\prime}}^{\prime}$ will be. For each term $w_{i}^{\prime}$ with $|w^{\prime}| - j + 1 \leq i \leq |w^{\prime}|$, the reduction forces $w_{i}^{\prime}$ to be $w_{i} - (\text{\# of terms in } w \text{ ``chopped off'' that were less than } w_{i})$. Thus, to compute $C(k,p_{1};[x_{1}, \ldots, x_{|p_{1}|}])$, we need to sum over all possible ways to ``fill out'' the rest of the terms in the final $p_{2}$ pattern of $w^{\prime}$. We also need to sum over all possible choices of $p_{2} \in B$ and all possible ways that the tails of this $p_{2}$ overlap with the heads of the final $p_{1}$ pattern.\\

In summary, the number of length $n$ permutations avoiding set $B$ can be found by, first, generating a cluster recurrence for $C(k,p;[x_{1}, \ldots, x_{|p|}])$ for each $p \in B$. Next, use the recurrence $\alpha(n) = n \alpha(n-1) + \mathop{\sum} \limits_{k=1}^{n} { {n \choose k} C(k) \alpha(n-k)}$ using the base cases $\alpha(0) = \alpha(1) = 1$ and $\alpha(n) = 0$ if $n < 0$. Use the cluster recurrences to compute $C(k)$ as needed:

\begin{eqnarray}
	C(k) & = & \mathop{\sum} \limits_{p \in B} C(k,p)\\
	C(k,p) & = & \mathop{\sum} \limits_{1 \leq x_{1} < x_{2} < \ldots < x_{m} \leq k} C(k,p; [x_{1}, \ldots, x_{m}])
\end{eqnarray}

Also recall that if $w = (\pi; [i_{1},j_{1}], \ldots, [i_{m},j_{m}])$, then $weight(w) = (t-1)^{m}$ will keep track of occurrences of patterns with variable $t$, while setting $t=0$ and using $weight(w) = (-1)^{m}$ would count only the permutations that avoid the designated pattern set $B$.

\subsection{Example}\label{RecurExSect}

Let $B = \{ 2 1 4 3 \}$. Let $w = (\pi; [[i_{1},j_{1}], \ldots, [i_{m},j_{m}]])$ be a length $k$ cluster and $\{ x_{1}, \ldots, x_{4} \}$ be the last $4$ terms of $w$ with $x_{1} < \ldots < x_{4}$ (i.e., $\pi_{k-3} = x_{2}$, $\pi_{k-2} = x_{1}$, $\pi_{k-1} = x_{4}$, and $\pi_{k} = x_{3}$). Then, the second to last pattern must also be a $2 1 4 3$ pattern and can have an overlap of length $1$ or $2$ with the last pattern.\\

If the overlap is of length $2$, let $\pi^{\prime} = \rho(\pi_{1} \cdots \pi_{k-2})$ and $w^{\prime} = (\pi^{\prime}; [[i_{1},j_{1}], \ldots, [i_{m-1},j_{m-1}]])$, the cluster found by ``chopping off'' the tail of the final bad pattern in $w$ and then canonically reducing. Now let $\{ y_{1}, \ldots, y_{4} \}$ be the last $4$ terms of $w^{\prime}$ with $y_{1} < \ldots < y_{4}$ (i.e., $\pi^{\prime}_{k-3} = y_{2}$, $\pi^{\prime}_{k-2} = y_{1}$, $\pi^{\prime}_{k-1} = y_{4}$, and $\pi^{\prime}_{k} = y_{3}$). Notice that the terms ``chopped off'' from $w$ were $x_{4}$ and $x_{3}$. Since both of these are larger than both $x_{2}$ and $x_{1}$, applying the reduction does not change their values. Thus, $y_{4} = x_{2}$ and $y_{3} = x_{1}$. Summing over all possible tails for $w^{\prime}$ and accounting for the last pattern that was removed from $w$, we get

\begin{eqnarray}
	\mathop{\sum} \limits_{\substack{1 \leq y_{1} < \ldots < y_{4} \leq k-2\\ y_{3} = x_{1}\\ y_{4} = x_{2}}} {weight(2 1 4 3) \cdot C(k-2;[y_{1}, y_{2}, y_{3}, y_{4}])}.
\end{eqnarray}

If the overlap is of length $1$, let $\pi^{\prime} = \rho(\pi_{1} \cdots \pi_{k-3})$ and $w^{\prime} = (\pi^{\prime}; [[i_{1},j_{1}], \ldots, [i_{m-1},j_{m-1}]])$, since the tail that gets ``chopped off'' has $3$ terms. Again, let $\{ y_{1}, \ldots, y_{4} \}$ be the last $4$ terms of $w^{\prime}$ with $y_{1} < \ldots < y_{4}$ (i.e., $\pi^{\prime}_{k-3} = y_{2}$, $\pi^{\prime}_{k-2} = y_{1}$, $\pi^{\prime}_{k-1} = y_{4}$, and $\pi^{\prime}_{k} = y_{3}$). The terms ``chopped off'' from $w$ are $x_{1}$, $x_{4}$, and $x_{3}$. Since exactly one term less than $x_{2}$ (only $x_{1}$) was removed, applying the reduction would reduce $x_{2}$ by $1$. Thus, $y_{3} = x_{2} - 1$. Summing over all possible tails for $w^{\prime}$ and accounting for the last pattern that was removed from $w$, we get

\begin{eqnarray}
	\mathop{\sum} \limits_{\substack{1 \leq y_{1} < \ldots < y_{4} \leq k-3\\ y_{3} = x_{2} - 1}} {weight(2 1 4 3) \cdot C(k-3;[y_{1}, y_{2}, y_{3}, y_{4}])}.
\end{eqnarray}

We combine the two possibilities along with the base cases to get the recurrence.\\

\noindent For $k < 4$:

\begin{eqnarray*}
	C(k;[x_{1}, x_{2}, x_{3}, x_{4}]) = 0
\end{eqnarray*}

\noindent For $k = 4$:

\begin{eqnarray*}
	C(k;[x_{1}, x_{2}, x_{3}, x_{4}]) = weight(2 1 4 3)
\end{eqnarray*}

\noindent For $k > 4$:

\begin{eqnarray}
	C(k;[x_{1}, x_{2}, x_{3}, x_{4}]) & = & \mathop{\sum} \limits_{\substack{1 \leq y_{1} < \ldots < y_{4} \leq k-3\\ y_{3} = x_{2} - 1}} {weight(2 1 4 3) \cdot C(k-3;[y_{1}, y_{2}, y_{3}, y_{4}])} \label{RecurEx}\\
	& & + \mathop{\sum} \limits_{\substack{1 \leq y_{1} < \ldots < y_{4} \leq k-2\\ y_{3} = x_{1}\\ y_{4} = x_{2}}} {weight(2 1 4 3) \cdot C(k-2;[y_{1}, y_{2}, y_{3}, y_{4}])} \notag
\end{eqnarray}

\noindent Using this recurrence, we can compute $C(k)$ for any value of $k$ and compute $\alpha(n)$ using Equation \eqref{AvRecur}. To keep track of all occurrence of $2 1 4 3$ with the variable $t$, let $weight(2 1 4 3) = t-1$. To only count permutations that avoid $2 1 4 3$, set $t=0$ so that $weight(2 1 4 3) = -1$ for the above recurrence.\\

\subsection{Results for c-Wilf-Equivalence}

Even though Section~\ref{CavAlg} is algorithmic in nature, it yields a strong theoretical byproduct. The cluster recurrence generated by the pattern set $B$ totally determines $\alpha(n)$. In fact, it also totally determines $P(z,t)$. However, the ``overlapping'' relations between the patterns in $B$ totally determines the cluster recurrence.\\

More specifically, let $B$ be the set of patterns we want to avoid, and let $\pi, \sigma \in B$ where $m = |\pi|$ and $n = |\sigma|$. Note that $\pi$ and $\sigma$ are not necessarily distinct. Suppose that $\rho(\sigma_{n - j + 1} \cdots \sigma_{n}) = \rho(\pi_{1} \cdots \pi_{j})$ (the tail of $\sigma$ and the head of $\pi$ has an overlap of length $j$). Then, define the following sets:

\begin{eqnarray}
	OverlapMap(\sigma,\pi,j) & = & \{ (\pi_{1},\sigma_{n - j + 1}), (\pi_{2},\sigma_{n-j+2}), \ldots, (\pi_{j},\sigma_{n})\}\\
	OverlapMaps(\sigma,\pi) & = & \{ OverlapMap(\sigma,\pi,j) : \rho(\sigma_{n - j + 1} \cdots \sigma_{n}) = \rho(\pi_{1} \cdots \pi_{j})\}
\end{eqnarray}\\

For example, in Section~\ref{RecurExSect}, the pattern $2143$ has self-overlaps of length $1$ and $2$. For a length $1$ overlap, we have $OverlapMap( 2143, 2143, 1) = \{(2,3)\}$. This combined with the length of the pattern, which is $4$, and the length of the original cluster, denoted by $k$, completely determines the first summation in Equation~\eqref{RecurEx}. Similarly, for a length $2$ overlap, we have $OverlapMap( 2143, 2143, 2) = \{(2,4), (1,3)\}$. Combining this with the length of the pattern, again $4$, and the length of the original cluster, again $k$, completely determines the second summation in Equation~\eqref{RecurEx}. Thus, $OverlapMaps(2143,2143) = \{ \{(2,3)\}, \{(2,4), (1,3)\} \}$ and $| 2143| = 4$ completely determines the cluster recurrence.\\

\noindent Therefore we have the following result based off of our algorithm:\\

\begin{thm}\label{OverlapThm}
	Let $B$ and $B^{\prime}$ be two sets of patterns with $|B| = |B^{\prime}|$. Suppose there is some labeling of the elements (patterns) in sets $B$ and $B^{\prime}$, say $B = \{ p_{1}, \ldots, p_{k} \}$ and $B^{\prime} =\{ p_{1}^{\prime}, \ldots, p_{k}^{\prime}\}$, such that $|p_{i}| = |p_{i}^{\prime}|$ for $1 \leq i \leq k$, and $OverlapMaps(p_{i},p_{j}) = OverlapMaps(p_{i}^{\prime},p_{j}^{\prime})$ for all $1 \leq i,j \leq k$. Then, $B$ and $B^{\prime}$ are strongly c-Wilf-Equivalent.
\end{thm}

\begin{proof}
	The cluster recurrence was uniquely determined by how the patterns overlapped (which terms from one pattern overlapped with which terms of another pattern) and by how they reduced after ``chopping'' the last pattern from the current cluster. The possible ways that two patterns can overlap are encoded by OverlapMaps and the effect of the reduction $\rho$ is determined by how the patterns overlapped and the length of those patterns.
\end{proof}

\noindent This result was also independently discovered by Khoroshkin and Shapiro, ~\cite{KS}.\\

Using this result, it is possible to classify c-Wilf-Equivalences in some cases. For example, it is possible to classify single pattern avoidance for single patterns of length $3$, $4$, and $5$ since all the potential equivalences that occur can be demonstrated using Theorem ~\ref{OverlapThm}. Using the same approach, we can nearly classify all single patterns of length $6$. All that remains are four possible strong c-Wilf-Equivalences that appear true but cannot be rigorously proven through our means. They are the following:\\

\begin{enumerate}
	\item[(1)] The pattern $1 2 3 5 4 6$ appears to belong to the strong c-Wilf-Equivalence class $\{1 2 4 5 3 6, 1 2 5 4 3 6\}$.\\
	\item[(2)] The pattern $1 2 3 6 4 5$ appears to belong to the strong c-Wilf-Equivalence class $\{1 2 4 6 3 5, 1 2 6 4 3 5\}$.\\
	\item[(3)] The patterns $1 3 2 4 6 5$ and $1 4 2 3 6 5$ appear to be strong c-Wilf-Equivalent.\\
	\item[(4)] The patterns $1 5 4 2 6 3$ and $1 6 5 2 4 3$ appear to be strong c-Wilf-Equivalent.
\end{enumerate}

\noindent The four cases have been experimentally verified for up to length $18$ permutations.\\

\subsection{Maple Implementation}

The algorithm from Section ~\ref{CavAlg} has been implemented in the Maple package {\tt CAV}. Using that algorithm, you can find the sequence $\alpha(1), \ldots, \alpha(n)$ avoiding a set of patterns $B$ by calling the procedure {\tt CAV(B,n)}, where the patterns in $B$ are represented as lists. For example, for $n=10$ and $B = \{ 1 2 3, 3 2 1\}$, trying {\tt CAV(\{[1,2,3],[3,2,1]\},10);} returns the output:
\begin{eqnarray*}
	[1, 2, 4, 10, 32, 122, 544, 2770, 15872, 101042]
\end{eqnarray*}

\noindent To keep track of the occurrences of patterns from $B$, use the procedure {\tt CAVt(B,n,t)}. For example, trying {\tt CAVt(\{[1,2,3],[3,2,1]\},6,t);} returns the output:
\begin{eqnarray*}
	[1, 2, 4+2t, 10+12t+2t^2, 58t+28t^2+32+2t^3, 300t+236t^2+122+60t^3+2t^4]
\end{eqnarray*}

\noindent Also, most of the main procedures in the Maple {\tt CAV} package have an optional verbose setting. For example, for the verbose outputs, try  {\tt CAV(\{[1,2,3],[3,2,1]\},10,true);}.\\


To generate the cluster recurrence only (encoded in a data structure that we call a scheme), use the procedure {\tt SCHEME(k,B,x,y,t)}. For example, try {\tt SCHEME(k,\{[1,2,3],[3,2,1]\},x,y,t);}. The overlap maps between two patterns can also be found using {\tt OverlapMaps(p1,p2)}, where you are checking for overlaps between tails of {\tt p1} with heads of {\tt p2}. For example, try {\tt OverlapMaps([2,1,4,3], [2,1,4,3]).}\\

To (attempt to) classify pattern sets of $m$ patterns with each pattern length $n$, we can compute $\alpha(N)$ (for some fixed value $N$) for each of these pattern sets, and if the $\alpha(N)$ values coincide, try to apply Theorem ~\ref{OverlapThm}. This has been implemented in the procedure {\tt WilfEqm(n,N,m)}. For example, try {\tt WilfEqm(5,12,1)} (or for the verbose output, {\tt WilfEqm(5,12,1,true)}) to (rigorously) classify c-Wilf-Equivalence for all single patterns of length $5$. An additional byproduct of Theorem ~\ref{OverlapThm} is that all instances of c-Wilf-Equivalence in single length $5$ patterns are actually strong c-Wilf-Equivalence. The $25$ c-Wilf-Equivalence classes can be found on the paper's website.\\

Similarly, we can use the {\tt WilfEqm} procedure to discover the following:

\begin{prop}
	Let $B_{1}$ and $B_{2}$ both be sets containing two patterns of length $3$. Then $B_{1}$ is c-Wilf-Equivalent to $B_{2}$ if and only if they are trivially equivalent by reversal and/or complementation.
\end{prop}

\begin{proof}
	Run ``{\tt WilfEqm(3,10,2,true);}'' using the {\tt CAV} Maple package.
\end{proof}

\begin{prop}
	Let $B_{1}$ and $B_{2}$ both be sets containing two patterns of length $4$. Then $B_{1}$ is c-Wilf-Equivalent to $B_{2}$ if and only if they are trivially equivalent by reversal and/or complementation.
\end{prop}

\begin{proof}
	Run ``{\tt WilfEqm(4,10,2,true);}'' using the {\tt CAV} Maple package.
\end{proof}

\begin{prop}
	Let $B_{1}$ and $B_{2}$ both be sets containing three patterns of length $3$. Then $B_{1}$ is c-Wilf-Equivalent to $B_{2}$ if and only if they are trivially equivalent by reversal and/or complementation.
\end{prop}

\begin{proof}
	Run ``{\tt WilfEqm(3,10,3,true);}'' using the {\tt CAV} Maple package.
\end{proof}

\noindent Similarly, nearly all c-Wilf-Equivalences could be classified for sets containing three patterns of length $4$. Four pairs of sets appear c-Wilf-Equivalent but cannot be proven through our means. They are the following:\\

\begin{enumerate}
	\item[(1)] The pattern sets $\{1 2 3 4, 1 2 4 3, 1 3 4 2\}$ and $\{1 2 3 4, 1 2 4 3, 1 4 3 2\}$ appear to be strongly c-Wilf-Equivalent.\\
	\item[(2)] The pattern sets $\{1 2 3 4, 1 2 4 3, 2 3 4 1\}$ and $\{1 2 3 4, 1 2 4 3, 2 4 3 1\}$ appear to be strongly c-Wilf-Equivalent.\\
	\item[(3)] The pattern sets $\{1 3 2 4, 1 3 4 2, 1 4 2 3\}$ and $\{1 3 2 4, 1 4 2 3, 1 4 3 2\}$ appear to be strongly c-Wilf-Equivalent.\\
	\item[(4)] The pattern sets $\{1 3 2 4, 1 4 2 3, 2 3 4 1\}$ and $\{1 3 2 4, 1 4 2 3, 2 4 3 1\}$ appear to be strongly c-Wilf-Equivalent.\\
\end{enumerate}

\noindent The four cases have been experimentally verified for up to length $14$ permutations, and the rest of the classification can be found on the paper's website.\\

\section{Consecutive Pattern Avoidance via the Cluster Tail Generating Function}\label{SectCAVT}

Computationally, the cluster recurrence is faster than the naive approach of checking every single permutation, but the approach is still very inefficient. For a fixed length $k$, not every combination of tails gives rise to a possible cluster. For example, if $B = \{ 1 2 3 \}$, the only possible underlying permutation in a length $9$ cluster is $1 2 3 4 5 6 7 8 9$. The only possible tail is $7 8 9$, but using the recurrence, we essentially try all $9 \choose 3$ possible tails. Each such possible tail gives its contribution of $0$ only after it has recursed down to the base cases of $k \leq 3$.\\

We can, however, gain a substantial speed-up by considering a more complicated generating function. For a fixed pattern $p \in B$ with length $m$, the cluster tail generating function will be defined as:

\begin{eqnarray}
	F(k,p; [z_{1}, \ldots, z_{m}]) = \mathop{\sum} \limits_{1 \leq x_{1} < \ldots < x_{m} \leq k} {C(k,p; [x_{1}, \ldots, x_{m}]) z_{1}^{x_{1}} \cdots z_{m}^{x_{m}} }\label{Fkpz}
\end{eqnarray}

\noindent If $B$ is a single pattern set, this may also be denoted as $F(k; [z_{1}, \ldots, z_{m}])$. Otherwise, we also define:

\begin{eqnarray}
	F(k; [z_{1}, \ldots, z_{m}]) = \mathop{\sum} \limits_{p \in B} {F(k,p; [z_{1}, \ldots, z_{m}])}
\end{eqnarray}

\noindent Note that $F(k,p; [1, \ldots, 1]) = C(k,p)$ and $F(k; [1, \ldots, 1]) = C(k)$. In fact, we can always find a functional equation for $F(k,p; [z_{1}, \ldots, z_{m}])$ of a certain form. We can then combine this with Equation \eqref{AvRecur} to more quickly compute $\alpha(n)$. We begin with an illustrative example and then present the general algorithm.\\

\subsection{Example}\label{CTGFExSect}

Let $B = \{1 3 2\}$ and suppose we want to only count permutations that completely avoid $1 3 2$. We will set $t=0$ which gives us $weight(1 3 2) = -1$. We then can find a functional equation for $F(k; [z_{1}, z_{2}, z_{3}])$ as follows. Using the procedure {\tt SCHEME} in the Maple package {\tt CAV}, we can get the following cluster recurrence:

\begin{eqnarray*}
	C(k;[x_{1}, x_{2}, x_{3}]) & = & -\mathop{\sum} \limits_{\substack{1 \leq y_{1} < y_{2} < y_{3} \leq k-2\\ y_{2}=x_{1}}} C(k-2;[y_{1}, y_{2}, y_{3}])\\
	& = & -\mathop{\sum} \limits_{\substack{1 \leq y_{1} < x_{1}\\ x_{1} < y_{3} \leq k-2}} C(k-2;[y_{1}, x_{1}, y_{3}])
\end{eqnarray*}

\noindent with the base cases $C(k;[x_{1}, x_{2}, x_{3}]) = 0$ if $k < 3$ and $C(k;[x_{1}, x_{2}, x_{3}]) = -1$ if $k = 3$. Substituting into Equation \eqref{Fkpz} and applying the finite geometric series formula as needed, we get:

\begin{eqnarray*}
	F(k;[z_{1}, z_{2}, z_{3}]) & = & \mathop{\sum} \limits_{1 \leq x_{1} < x_{2} < x_{3} \leq k} {C(k; [x_{1}, x_{2}, x_{3}]) z_{1}^{x_{1}} z_{2}^{x_{2}} z_{3}^{x_{3}} }\\
	& = & - \mathop{\sum} \limits_{x_{1}=1}^{k-2} \mathop{\sum} \limits_{x_{2}=x_{1}+1}^{k-1} \mathop{\sum} \limits_{x_{3}=x_{2}+1}^{k} \mathop{\sum} \limits_{\substack{1 \leq y_{1} < x_{1}\\ x_{1} < y_{3} \leq k-2}} C(k-2;[y_{1}, x_{1}, y_{3}]) z_{1}^{x_{1}} z_{2}^{x_{2}} z_{3}^{x_{3}}\\
	& = & - \mathop{\sum} \limits_{x_{1}=1}^{k-2} \mathop{\sum} \limits_{x_{2}=x_{1}+1}^{k-1} \mathop{\sum} \limits_{\substack{1 \leq y_{1} < x_{1}\\ x_{1} < y_{3} \leq k-2}} C(k-2;[y_{1}, x_{1}, y_{3}]) z_{1}^{x_{1}} z_{2}^{x_{2}} \mathop{\sum} \limits_{x_{3}=x_{2}+1}^{k} z_{3}^{x_{3}}\\
	& = & - \frac{z_{3}}{1-z_{3}} \mathop{\sum} \limits_{x_{1}=1}^{k-2} \mathop{\sum} \limits_{x_{2}=x_{1}+1}^{k-1} \mathop{\sum} \limits_{\substack{1 \leq y_{1} < x_{1}\\ x_{1} < y_{3} \leq k-2}} C(k-2;[y_{1}, x_{1}, y_{3}]) z_{1}^{x_{1}} z_{2}^{x_{2}} (z_{3}^{x_{2}} - z_{3}^{k})\\
	& = & - \frac{z_{3}}{1-z_{3}} \mathop{\sum} \limits_{x_{1}=1}^{k-2} \mathop{\sum} \limits_{\substack{1 \leq y_{1} < x_{1}\\ x_{1} < y_{3} \leq k-2}} C(k-2;[y_{1}, x_{1}, y_{3}]) z_{1}^{x_{1}} \mathop{\sum} \limits_{x_{2}=x_{1}+1}^{k-1} z_{2}^{x_{2}} (z_{3}^{x_{2}} - z_{3}^{k})
\end{eqnarray*}

\noindent and since 

\begin{eqnarray*}
	\mathop{\sum} \limits_{x_{2}=x_{1}+1}^{k-1} z_{2}^{x_{2}} (z_{3}^{x_{2}} - z_{3}^{k}) = \left( \frac{(z_{2} z_{3})^{x_{1}+1} - (z_{2} z_{3})^{k}}{1 - z_{2} z_{3}} - z_{3}^{k} \frac{z_{2}^{x_{1}+1} - z_{2}^{k}}{1 - z_{2}} \right)
\end{eqnarray*}

\noindent we get

\begin{eqnarray*}
	F(k;[z_{1}, z_{2}, z_{3}]) & = & - \frac{z_{2} z_{3}^{2}}{(1-z_{3}) (1 - z_{2} z_{3})} \mathop{\sum} \limits_{x_{1}=1}^{k-2} \mathop{\sum} \limits_{\substack{1 \leq y_{1} < x_{1}\\ x_{1} < y_{3} \leq k-2}} C(k-2;[y_{1}, x_{1}, y_{3}]) (z_{1} z_{2} z_{3})^{x_{1}} \\
	& & + \frac{z_{2}^{k} z_{3}^{k+1}}{(1-z_{3}) (1 - z_{2} z_{3})} \mathop{\sum} \limits_{x_{1}=1}^{k-2} \mathop{\sum} \limits_{\substack{1 \leq y_{1} < x_{1}\\ x_{1} < y_{3} \leq k-2}} C(k-2;[y_{1}, x_{1}, y_{3}]) z_{1}^{x_{1}}\\
	& & + \frac{z_{2} z_{3}^{k+1}}{(1 - z_{3}) (1 - z_{2})} \mathop{\sum} \limits_{x_{1}=1}^{k-2} \mathop{\sum} \limits_{\substack{1 \leq y_{1} < x_{1}\\ x_{1} < y_{3} \leq k-2}} C(k-2;[y_{1}, x_{1}, y_{3}]) (z_{1} z_{2})^{x_{1}}\\
	& & - \frac{z_{2}^{k} z_{3}^{k+1}}{(1 - z_{3}) (1 - z_{2})} \mathop{\sum} \limits_{x_{1}=1}^{k-2} \mathop{\sum} \limits_{\substack{1 \leq y_{1} < x_{1}\\ x_{1} < y_{3} \leq k-2}} C(k-2;[y_{1}, x_{1}, y_{3}]) z_{1}^{x_{1}}\\
	& = & - \frac{z_{2} z_{3}^{2}}{(1-z_{3}) (1 - z_{2} z_{3})} F(k-2; [1, z_{1} z_{2} z_{3}, 1])\\
	& & + \frac{z_{2}^{k} z_{3}^{k+1}}{(1-z_{3}) (1 - z_{2} z_{3})} F(k-2;[1, z_{1}, 1])\\
	& & + \frac{z_{2} z_{3}^{k+1}}{(1 - z_{3}) (1 - z_{2})} F(k-2;[1, z_{1} z_{2},1])\\
	& & - \frac{z_{2}^{k} z_{3}^{k+1}}{(1 - z_{3}) (1 - z_{2})} F(k-2;[1,z_{1},1]).
\end{eqnarray*}

We can then use the functional equation to compute $C(k) = F(k;[1,1,1])$ for whatever $k$ we need and then find $\alpha(n)$ for the desired $n$ by Equation \eqref{AvRecur}.\\

\subsection{General Algorithm}\label{TailAlg} 

In general, if we can find a functional equation for $F(k,p; [z_{1}, \ldots, z_{m}])$ that relates it to cluster generating functions with lower order `$k$', we can use it to compute $\alpha(n,t)$ using Equation \eqref{AvRecur}. One can see that most of what was done in the above example can be extended to any pattern (or pattern set by finding a functional equation for each pattern individually). The outline of the general procedure is as follows:\\

First, find the cluster recurrence for the initial summand $C(k,p; [x_{1}, \ldots, x_{m}])$ (as in Section ~\ref{CavAlg}) and substitute this into the summation in Equation \eqref{Fkpz}. Split the summation over each summand $C(k^{\prime},p^{\prime}; [y_{1}, \ldots, y_{m^{\prime}}])$, and handle each one separately. Rewrite the summations over $x_{1}, \ldots, x_{m}$ and apply the finite geometric series formula as needed. Finally, express the remaining summations as cluster tail generating functions of lower order $k^{\prime}$. \\

The only part that is not immediate is whether the summations for $x_{1}, \ldots, x_{m}$ can be ordered properly and whether the lower and upper bounds for each summation index can be chosen properly so that we can adequately apply the finite geometric series formula. This can in fact always be done, and the ordering and choice of bounds can be done as follows:\\

Let $x_{i_{1}}, \ldots, x_{i_{j}}$ be the entries from the original last pattern $p$ in the length $k$ cluster that coincide with entries from the new last pattern $p^{\prime}$ in the length $k^{\prime}$ cluster. In other words, $x_{i_{1}}, \ldots, x_{i_{j}}$ are the terms that occur in the $y_{i}$'s of  $C(k^{\prime},p^{\prime}; [y_{1}, \ldots, y_{m^{\prime}}])$ . Let $x_{i_{j+1}}, \ldots, x_{i_{m}}$ be the terms that were ``chopped off'' from the length $k$ cluster. Also, assume that $x_{i_{1}} < \ldots < x_{i_{j}}$ and $x_{i_{j+1}} < \ldots < x_{i_{m}}$. Note that in the example in Section ~\ref{CTGFExSect}, $x_{i_{1}} = x_{1}$ (not ``chopped'') while $x_{i_{2}} = x_{2}$ (``chopped'') and $x_{i_{3}} = x_{3}$ (``chopped'').\\

\noindent \textbf{Order of summations:}\\
The summations will be ordered (from outermost to innermost) as $x_{i_{1}}$ to $x_{i_{j}}$ followed by $x_{i_{j+1}}$ to $x_{i_{m}}$. Thus, the outermost summation is indexed by $x_{i_{1}}$, the next summation inward is indexed by $x_{i_{2}}$, and so on. This places the summations over $x_{i_{j+1}}, \ldots, x_{i_{m}}$ to be on the ``inside'' so that they can be moved inward to apply the finite geometric series formula.\\

\noindent \textbf{Lower/Upper bounds for $x_{i_{1}}, \ldots, x_{i_{m}}$:}\\
\noindent For each $l$ with $j+1 \leq l \leq m$, let $b_{l} = k$ if $i_{l} > i_{j}$; otherwise, let $b_{l} = \min(\{ i_{1}, \ldots, i_{j}\} \backslash \{1, \ldots, i_{l} \})$, and let $c_{l}$ be the index of $i$ (so $i_{c_{l}} = b_{l}$).\\

\noindent \textit{For $x_{i_{1}}, \ldots, x_{i_{j}}$}:

\begin{eqnarray*}
	x_{i_{1}} & = & i_{1} \text{ to } k-m+i_{1}\\
	x_{i_{2}} & = & x_{i_{1}}+i_{2}-i_{1} \text{ to } k-m+i_{2}\\
	& & \cdots\\
	x_{i_{j}} & = & x_{i_{j-1}}+i_{j}-i_{j-1} \text{ to } k-m+i_{j}
\end{eqnarray*}

\noindent \textit{For $x_{i_{j+1}}$}:\\
\noindent Lower bound is $1$ if $i_{j+1} = 1$, and $x_{i_{j+1}-1}+1$ otherwise. Upper bound is $k-m+i_{j+1}$ if $b_{j+1}=k$, and $b_{j+1}-c_{j+1}+i_{j+1}$ otherwise.\\

\noindent \textit{For $x_{i_{l}}$ with $l>j+1$}:\\
\noindent Lower bound is $x_{i_{l}-1}+1$. Upperbound is $k-m+i_{l}$ if $b_{l}=k$, and $b_{l}-c_{l}+i_{l}$ otherwise.\\

One can see that the indices $x_{i_{1}}, \ldots, x_{i_{j}}$ range over all necessary values and can also verify that $x_{i_{j+1}}, \ldots, x_{i_{m}}$ will cover all necessary values as well. Additionally, for each $r$, the lower and upper bounds for $x_{i_{r}}$ never depends on any $x_{i_{s}}$ where $s > r$. If we applied the above approach to the example in Section ~\ref{CTGFExSect}, we would get $x_{i_{1}} = x_{1}$ going from $1$ to $k-2$, $x_{i_{2}} = x_{2}$ going from $x_{1} + 1$ to $k-1$, and $x_{i_{3}} = x_{3}$ going from $x_{2} + 1$ to $k$.\\

\subsection{Additional Results}

We get a couple more immediate byproducts from the algorithm in Section \ref{TailAlg}. First, the method provided for finding a functional equation always works, so we get the following:

\begin{thm}
	Let $B$ be a pattern set and $p \in B$. Then, there always exists a functional equation for $F(k,p;[z_{1}, \ldots, z_{|p|}])$ of the form:
	\begin{eqnarray*}
		F(k,p;[z_{1}, \ldots, z_{|p|}]) = (t-1) \mathop{\sum} \limits_{p^{\prime} \in B} \mathop{\sum} \limits_{i \in I(p^{\prime})} {R_{i} \cdot F(k_{i},p^{\prime}; [M^{i}_{1}, \ldots, M^{i}_{|p^{\prime}|}])}
	\end{eqnarray*}
	where $I(p^{\prime})$ is a finite index set for each $p^{\prime} \in B$, $I(p^{\prime})$ and $I(p^{\prime \prime})$ are disjoint if $p^{\prime} \neq p^{\prime \prime}$, each $M^{i}_{j}$ is a specific monomial in $z_{1}, \ldots, z_{|p|}$, each $R_{i}$ is a specific rational expression in $z_{1}, \ldots, z_{|p|}$, and $k_{i} < k$ for each $i$.
\end{thm}

\noindent Additionally, we get an immediate corollary of Theorem~\ref{ThmGJ}.

\begin{cor}
	Let $B$ be a set of patterns we would like to avoid. Without loss of generality, assume that $B$ contains no redundancies. Then by setting $weight(p) = t-1$ for each $p \in B$, we get:
	\begin{eqnarray*}
		P(z,t) = \frac{1}{1 - z - \mathop{\sum} \limits_{k \geq 1} {\mathop{\sum} \limits_{p \in B} {F(k,p;[1, \ldots, 1]) \frac{z^{k}}{k!}}} }.
	\end{eqnarray*}
\end{cor}

\noindent Given that we can find a functional equation given any pattern set $B$, in a sense, we have an expression for the exponential generating function $P(z,t)$ for any pattern set.

\subsection{Maple Implementation}

The algorithm from Section \ref{TailAlg} has also been implemented in the Maple package {\tt CAV}. Using that algorithm, you can find the sequence $\alpha(1), \ldots, \alpha(n)$ avoiding a set of patterns $B$ by calling the procedure {\tt CAVT(B,n)}, where the patterns in $B$ are represented as lists. For example, for $n=10$ and $B = \{ 1 2 3, 3 2 1\}$, try {\tt CAVT(\{[1,2,3],[3,2,1]\},10);}. To keep track of the occurrences of patterns from $B$, use the procedure {\tt CAVTt(B,n,t)}. For example, try {\tt CAVTt(\{[1,2,3],[3,2,1]\},10,t);}. To generate the cluster tail functional equation only (encoded again in a data structure that we call a scheme), use the procedure {\tt MakeTailFE(B,k,z,t)}. For example, try {\tt MakeTailFE(\{[1,3,2]\},k,z,t);}. \\

Computationally, the algorithm in Section \ref{TailAlg} is much more efficient than the one in Section \ref{CavAlg}, so the {\tt CAVT} procedure is much faster than the {\tt CAV} procedure. In general, {\tt CAVT} should be used instead of {\tt CAV} for computing $\alpha(n)$ values and, similarly, {\tt CAVTt} should be used instead of {\tt CAVt} for $\alpha(n,t)$.

\section{Asymptotic Approximations Using CAV}\label{SectAsym}

Let $B = \{ p \}$ be a set containing a single pattern. In ~\cite{RW}, Warlimont gave a conjecture on the asymptotics of $\alpha(n)$:

\begin{eqnarray}
	\alpha(n) \sim \gamma \cdot \rho^{n} \cdot n!
\end{eqnarray}\\

\noindent where $\gamma$ and $\rho$ are constants depending only on the single pattern $p$. Some initial asymptotic results for $\alpha(n)$ were proven by Elizalde in ~\cite{SE}. Recently, Ehrenborg, Kitaev, and Perry prove this conjecture in ~\cite{EKP}. With this result established, we can compute approximate values of $\gamma$ and $\rho$ for various single patterns.\\

Elizalde and Noy gave some approximations of $\gamma$ and $\rho$ for length $3$  and a few length $4$ patterns in ~\cite{EN}. Aldred, Atkinson, and McCaughan also gave approximations for the $\rho$ values of the single length $4$ patterns. Using the Maple package {\tt CAV}, we can empirically verify these approximations and also quickly produce many new approximations. For example, the procedure {\tt AsymApprox(p,N,d)} will give approximate values (up to $d$ decimal digits) for $\gamma$ and $\rho$ for the pattern $p$ by computing $\alpha(N-2)$, $\alpha(N-1)$, and $\alpha(N)$ and computing their ratios. For example, try {\tt AsymApprox([1,2,4,3],50,20)}.\\

To approximate $\gamma$ and $\rho$ values (up to $d$ decimal digits) for all length $n$ patterns and then rank them by the size of $\rho$, use {\tt AsymApproxRank(n,N,d)}. For example, {\tt AsymApproxRank(4,30,10)} gives us the approximations for the $\gamma$ and $\rho$ values for length $4$ patterns:\\

\begin{table}[h!b!p!]
	\begin{tabular}{|c | c | c|}
		\hline
		Pattern & $\gamma$ & $\rho$\\ \hline
		1 2 3 4 & 1.1176930011 & 0.9630055289\\ \hline
		2 4 1 3 & 1.1375931232 & 0.9577180134\\ \hline
		2 1 4 3 & 1.1465405299 & 0.9561742431\\ \hline
		1 3 2 4 & 1.1510444988 & 0.9558503134\\ \hline
		1 4 2 3 & 1.1567436851 & 0.9548260509\\ \hline
		1 3 4 2 $\sim$ 1 4 3 2 & 1.1561985648 & 0.9546118344\\ \hline
		1 2 4 3 & 1.1696577874 & 0.9528914233\\ 
		\hline
	\end{tabular}\\
	\caption{Approximate asmptotics for length 4 patterns}
\end{table}

\noindent Similarly, {\tt AsymApproxRank(5,25,20)} would give us the approximations for the $\gamma$ and $\rho$ values for length $5$ patterns. The output can be found on the paper's website.\\

\section{Further Work}\label{SectConj}

In this paper, we outlined the key procedures in the {\tt CAV} Maple package. The cluster tail generating function was defined, and a constructive approach was demonstrated in finding a functional equation for it. Using this functional equation, we were able to more quickly count permutations avoiding a prescribed set of patterns. In addition, by applying Theorem ~\ref{OverlapThm}, we were able to totally classify c-Wilf-Equivalences in single patterns of length $3$, $4$, and $5$ rigorously, while nearly classifying single patterns of length $6$. We were also able to classify c-Wilf-Equivalences in a few cases of multiple pattern sets. Finally, we were able to use the faster algorithm to compute approximate values for asymptotic constants.\\

Despite this, there is a lot of room for improvement algorithmically and quite a few new open problems/conjectures arise. Some of the conjectures are listed below.\\

\noindent Elizalde and Noy provided the following conjecture in ~\cite{EN}:

\begin{conj}
	For a fixed pattern length $k$, the increasing pattern $\sigma = 1 2 \ldots k$ is the ``maximal'' pattern, in the sense that $\alpha_{\sigma}(n) \geq \alpha_{p}(n)$ for all $p \in S_{k}$ and all $n$.\\
\end{conj}

\noindent Based off of experimentation, we also have the following analogous conjectures:

\begin{conj}
	For a fixed pattern length $k$, the pattern $\sigma = 1 2 \ldots (k-2) (k) (k-1)$ is the ``minimal'' pattern, in the sense that $\alpha_{p}(n) \geq \alpha_{\sigma}(n)$ for all $p \in S_{k}$ and all $n$.
\end{conj}

\begin{conj}
	For a fixed pattern length $k$, the pattern set $B = \{1 2 \ldots k, 2 3 \ldots k 1\}$ is the ``maximal'' pattern set among sets of $2$ patterns, in the sense that $\alpha_{B}(n) \geq \alpha_{B^{\prime}}(n)$ for all $B^{\prime} \in {S_{k} \choose 2}$ and all $n$.
\end{conj}

\begin{conj}
	For a fixed pattern length $k$, the pattern set $B = \{1 2 \ldots (k-2) (k) (k-1), 1 2 \ldots (k-3) (k-1) (k) (k-2)\}$ is the ``minimal'' pattern set among sets of $2$ patterns, in the sense that $\alpha_{B^{\prime}}(n) \geq \alpha_{B}(n)$ for all $B^{\prime} \in {S_{k} \choose 2}$ and all $n$.
\end{conj}

\begin{conj}
	For a fixed pattern length $k$, the pattern set $B = \{1 2 \ldots k, 2 3 \ldots k 1, k 1 2 \ldots (k-1)\}$ is the ``maximal'' pattern set among sets of $3$ patterns, in the sense that $\alpha_{B}(n) \geq \alpha_{B^{\prime}}(n)$ for all $B^{\prime} \in {S_{k} \choose 3}$ and all $n$.\\
\end{conj}

\noindent In addition, based off of empirical evidence for single pattern avoidance up to length $6$ patterns, we believe the following:

\begin{conj}
	For any two patterns $p_{1}$ and $p_{2}$ of the same length, either $p_{1}$ and $p_{2}$ are strongly c-Wilf-Equivalent or they are not c-Wilf-Equivalent at all. 
\end{conj}

\noindent This certainly holds for single patterns of length $3$, $4$, and $5$. This also appears to hold for single length $6$ patterns.\\

\noindent Finally, the author would like to thank Doron Zeilberger for his suggestions, comments, and encouragement towards the work in this paper.\\

\end{document}